\documentclass[11pt]{amsart}
\usepackage{amsbsy,amssymb,amscd,amsfonts,latexsym,amstext,delarray, amsmath,graphicx,color,comment,wrapfig}%,showkeys}
\usepackage{graphicx}
\input xy

\xyoption{all}
\usepackage{euscript}

\newcommand{\F}{\mathbb{F}}
\newcommand{\G}{\Gamma}
\newcommand{\E}{\mathbb{E}}
\newcommand{\La}{\Lambda}
\newcommand{\Ga}{\Gamma}

\newcommand{\Z}{\mathbb{Z}}
\newcommand{\SL}{\operatorname{SL}}
\newcommand{\diag}{\operatorname{diag}}
\newcommand{\GL}{\operatorname{GL}}
\newcommand{\PGL}{\operatorname{PGL}}
\newcommand{\Stab}{\operatorname{Stab}}

\newtheorem{theorem}{Theorem}[section]
\newtheorem{prop}[theorem]{Proposition}
\newtheorem{lemma}[theorem]{Lemma}

\newtheorem{remark}[theorem]{Remark}

\newtheorem{conjecture}[theorem]{Conjecture} 
\begin{document}

\title{\bf{Fundamental domains for congruence subgroups of $\SL_2$\\
in positive characteristic}}

\date{\today}

\author{Lisa Carbone, Leigh Cobbs and Scott H.\ Murray}

\thanks{The first author was supported in part by NSF grant \#DMS-0701176.\\
Primary AMS subject codes: 20E08, 05C25, 20-04; Secondary AMS subject code: 20F32}

\address{Department of Mathematics, Hill Center, Busch Campus\\
Rutgers, The State University of New Jersey\\
110 Frelinghuysen Rd\\
Piscataway, NJ 08854-8019}
\email{carbonel@math.rutgers.edu}
\address{Department of Mathematics, Hill Center, Busch Campus\\
Rutgers, The State University of New Jersey\\
110 Frelinghuysen Rd\\
Piscataway, NJ 08854-8019}
\email{cobbs@math.rutgers.edu}
\address{Discipline of Mathematics and Statistics\\
Faculty of Information Science and Engineering\\
University of Canberra, ACT 2601, Australia}
\email{murray@maths.usyd.edu.au}

\thanks{We are indebted to Gunther Cornelissen for extremely helpful discussions which led us to the completion of this work. We are also grateful to Gunther for notifying us of the work of Gekeler-Nonnengardt and Rust, and for referring us to Max Gebhardt. We thank Max Gebhardt for providing us with the details of his unpublished computations. We are extremely grateful to Ernst-Ulrich Gekeler for his careful reading of our manuscript which led to a number of improvements and for clarifying the details of this earlier work.  We thank Dimitri Leemans for helping us with theory and computation for coset graphs.}

\begin{abstract}  In this work, we construct  fundamental domains for congruence subgroups of $\SL_2(\mathbb{F}_q[t])$ and $\PGL_2(\mathbb{F}_q[t])$. Our method uses Gekeler's description of the fundamental domains on the Bruhat-Tits tree $X=X_{q+1}$ in terms of cosets of subgroups. We compute the fundamental domains for a number of congruence subgroups explicitly as graphs of groups using the computer algebra system Magma. 
%Our graphs coincide with Morgenstern's quotient graphs arising from congruence subgroups of $\PGL_2$.

\end{abstract}

\maketitle

\section{Introduction}\label{sec:intro}

We construct fundamental domains   for congruence subgroups of the group $\Gamma=\SL_2(\mathbb{F}_q[t])$ which is a nonuniform lattice subgroup of $G=\SL_2(\mathbb{F}_q((t^{-1})))$. These congruence subgroups have the form $ \Gamma(g)=\{A \in\SL_2(\mathbb{F}_q[t])\mid A \equiv I_2 \bmod g \}$
for some $g\in\mathbb{F}_q[t]$.  Our method is to explicitly construct the fundamental domain for $\Gamma(g)$ as a graph which is a `ramified covering' of the quotient graph for $\Gamma$ on the Bruhat-Tits  tree $X=X_{q+1}$ of $G$. 
This approach is consistent with the theory of branched topological coverings and coincides with a method suggested by  Drinfeld in his theory of modular curves over function fields \cite{D}. This method of fibering the graphs 
$\Gamma(g)\backslash X$ over $\Gamma\backslash X$ and describing the vertices and edges of $\Gamma(g)\backslash X$ as suitable cosets in $\Gamma$ first appeared in the doctoral thesis of Gekeler \cite{Ge1} (see also \cite{Ge2}). Gekeler and Nonnengardt \cite{GN} and Rust \cite{Ru} have also given independent constructions of fundamental domains of lattices for congruence subgroups. 

%Some of our results on the  construction of graphs of groups for congruence subgroups of $\SL_2$ are included in these works. In particular, it follows from \cite{Ge1} and \cite{Ge2} that our ramified covering for  $\Gamma(g)$ on $X$ over $\Gamma\backslash X$ coincides with the quotient graph $\Gamma(g)\backslash X$ (see Section 3 for further discussion). 

The structural properties of the quotient graphs obtained as ramified coverings are nontrivial to determine. We use the Magma computer algebra system \cite{magma} to construct explicit examples. This involves a number of advanced features of Magma  including finite matrix groups, graph isomorphism \cite{nauty}, and finite geometries \cite{Leemans}. We drew some of the resulting graphs with the program {\tt dot} which is part of the Graphviz graph visualization system \cite{graphviz}.

Our initial motivation for this work was to obtain  explicit examples of  Morgenstern's construction of fundamental domains for congruence subgroups of the lattice $\Gamma=\PGL_2(\mathbb{F}_q[t])$ \cite{M}.  Let $X_g$ denote the quotient graph of the Bruhat-Tits tree $X=X_{q+1}$ by $\Gamma(g)$.  Morgenstern proved that certain subgraphs of $X_g$ provide the first known examples of linear families of bounded concentrators \cite{M}.  In the abstract and Section 3 of his paper  \cite{M}, he also gives a construction of the graph $X_g$ in terms of cosets, following the method of Gekeler.  We have explicitly constructed these coset graphs, and found that they are disconnected in characteristic 2, and so cannot be quotient graphs by the action of congruence subgroups on the Bruhat-Tits tree. Moreover the subgraphs at levels $0-1$, which he claims are bounded concentrators, are also not connected in characteristic 2.  We believe that his error is confined to the construction of $X_g$ as a coset graph, and does not effect his main results. We clarify the construction of Morgenstern and we prove that his full graphs are connected only in odd characteristic (Sections 4 and 5).  

We mention the following related results. After preparation of this manuscript, we learned that independent computations by  Max Gebhardt \cite{G} also show that Morgenstern's graphs  are not connected.  Chris Hall has notified us that he wrote an explicit algorithm for constructing fundamental domains based on earlier work of Gekeler \cite{Ha}. The Master's Thesis of Ralf Butenuth \cite{Bu} contains a construction of arbitrary congruence subgroups of $\PGL_2(\mathbb{F}_q[t])$. He also implemented an algorithm, using sieving methods but no advanced Magma functions, to compute the quotient graphs of Bruhat-Tits trees by congruence subgroups \cite{Bu}.

\section{Fundamental domains for congruence subgroups of $\SL_2(\F_q[t])$\\ as ramified coverings}

In this section we give a construction of quotient graphs for congruence subgroups of $\Gamma = \SL_2(\mathbb{F}_q[t])$ acting on the Bruhat-Tits tree $X=X_{q+1}$ of  $G= \SL_2(\mathbb{F}_q((t^{-1})))$ as ramified coverings over $\Gamma\backslash X$.  This construction is in terms of cosets, following the method of Gekeler \cite{Ge1,Ge2}.

\subsection{Ramified Coverings}
Our graphs are  connected, oriented  and locally finite. A  tree is a nonempty graph without closed circuits. Suppose a group $\Gamma$ acts on a tree $X$ without inversions. Then the quotient graph $\Gamma\,\backslash X$ is well defined and there is a natural quotient morphism $X\longrightarrow \Gamma\,\backslash X$.  Given a normal subgroup $N$ in $\G$, we can define the quotient graph $N\, \backslash X$  by
$$
 V(N\backslash X)  =  N\, \backslash V(X)=\{N\cdot v \mid  v\in V(X) \},\qquad  E(N\backslash X) = N\, \backslash E(X)=\{N\cdot e \mid e\in E(X) \} .
$$
Then $\Gamma / N$ acts on $N\backslash X$ by $ \gamma N(N\cdot x) = N \cdot \gamma x,$ where $x$ denotes either a vertex or an edge of $X$. Equivalently, we can take the graph $N\,\backslash X$ to have vertices (respectively edges) given by cosets of $\Stab_\Gamma(x)\Gamma/N$ in $\Gamma/N$, $x\in V(\Gamma\,\backslash X)$ (respectively cosets of $\Stab_\Gamma(e)\Gamma/N$ in $\Gamma/N$, $e\in E(\Gamma\,\backslash X)$). 
Cosets are adjacent as vertices in the graph $N\,\backslash X$ if and only if their intersection is non-empty. We call this construction of $N\,\backslash X$ a {\it ramified covering} over $\Gamma\,\backslash X$.

\subsection{Fundamental domain for $\G=\SL_2(\F_q[t])$}
Let $\Gamma= \SL_2(\mathbb{F}_q[t])\leq G= \SL_2(\mathbb{F}_q((t^{-1})))$.  
The Bruhat-Tits building of $G$ is the $(q+1)$-homogeneous tree $X=X_{q+1}$ \cite{S}. Serre~\cite{S} gives the fundamental domain for $\Gamma=\SL_2(\F_q[t])$ on $X$ as a semi-infinite ray \cite[Proposition 3, p.\ 87]{S}. We construct fundamental domains for congruence subgroups of $\G$ as ramified coverings over $\Gamma\backslash X$.  Since these subgroups are normal, there is also an action by the quotient groups.

Let $\Gamma$ be a group and $X$ a tree.  Suppose $\Gamma$ acts on 
$X$.  If $N$ is a normal subgroup of $\Gamma$, then $\Gamma/N$ acts on the 
connected graph $N\,\backslash X$.  Each $Nx\in N\,\backslash X$ has stabilizer 
$ \Stab_{\Gamma/N}(Nx) =  N\Stab_\Gamma(x)/N.$

Therefore, given a normal subgroup $N$ of $\G$, we may describe the vertices (respectively edges) of $N\,\backslash X$ not only as $N$-orbits with respect to the action of $N$ on $X$, but as $\Gamma/N$-orbits of $\{Nv : v\in V(\Gamma \backslash X) \}$ (respectively of $\{ Ne : e\in E(\Gamma \backslash X) \}$).

\subsection{Levelled coset graphs}
Let $H$ be a group and let $H_0,H_1,H_2,\dots$ be a (finite or infinite) 
sequence of subgroups of $H$.
We define the \emph{levelled coset graph} given by 
$H_0,H_1,\ldots\le H$ as follows:
The vertex set is partitioned into levels $L_0,L_1,\ldots$, with
vertices at level $i$ corresponding to cosets $hH_i$, for $h\in H$.
There is an edge connecting $hH_i$ with $kH_{i+1}$ if, and only if $hH_i \cap kH_{i+1}  \neq \emptyset.$ 
There are no edges between vertices in non-adjacent levels.

It is easy to show that the edges between levels $i$ and $i+1$ correspond to the
the cosets of $H_i\cap H_{i+1}$.  The edge connecting $hH_i$ to $kH_{i+1}$ corresponds to
$jH_i\cap H_{i+1}$, for some $j$ in the intersection of $hH_i$ and $kH_{i+1}$.

We consider levelled coset graphs with $H_1\le H_2\le\cdots$. The following proposition is a slight modification of a standard result for coset graphs:
\begin{prop}\label{P-gen}
The levelled coset graph given by $H_0,H_1,\dots,H_{n-1}\le H$
with $$H_1\le H_2\le\cdots\le H_{n-1}$$ 
has $|H:\langle H_0,H_{n-1}\rangle|$ connected components.
\end{prop}
%\begin{proof}
%Suppose that $H=\langle H_0,H_{n-1}\rangle$.
%Clearly there is a path connecting $H_ia$ and $H_ja$ for every $i,j=0,\dots,n-1$.
%Let $H_ia$ and $H_jb$ be two vertices, with $a$ and $b$ in the same coset of $\langle H_0,H_{n-1}\rangle$.
%Write
%$$
%  ba^{-1}=h_1k_1h_2k_2\cdots h_mk_m
%$$
%for $h_l\in H_0$ and $k_l\in H_{n-1}$.
%Then we have a path from $H_ia$ to $H_{n-1}a=H_{n-1}k_ma$, to
%$H_0k_ma=H_0h_mk_ma$, to $H_{n-1}h_ma=H_{n-1}k_{m-1}h_mk_ma$,
%and so on, to $H_0h_1k_1\cdots h_mk_ma=H_0b$, and so to $H_jb$.

%Conversely, given a path from $H_ia$ and $H_jb$, we can construct a path
%from  $H_0a$ to $H_0b$ since $H_1\le H_2\le\cdots$.
%This second path gives us $a^{-1}b\in H$ as a word in elements of the groups $H_0$ and $H_{n-1}$.
%\end{proof}

\subsection{The levels of $X_g$}
Fix $g\in \F_q[t]$ of degree $n$.
Since $ \Gamma(g)=\{A \in\PGL_2(\mathbb{F}_q[t])\mid A \equiv I_2 \bmod g \}$
is normal in $\G$, the quotient graph $X_g=\G(g) \backslash X$ may be viewed as a ramified covering of the quotient graph $\G\, \backslash X$, which is a semi-infinite ray. We may partition $V(X)$ into 
$\G$-orbits of the $\La_i$.  
Since $\G_i=\Stab_\G(\La_i)$ by \cite[Proposition 3, p.\ 87]{S} the 
orbit $\G\cdot\La_i$ is in one-to-one correspondence with $\G/\G_i$.  
Similarly the edges between $\G\cdot\La_i$ and  $\G\cdot\La_{i+1}$
correspond to $\G/(\G_i \cap \G_{i+1})$.
So we make the identifications
\begin{align*}
V(X) &= \bigsqcup_{i\geq 0} \G/\G_i, &
E(X) &= \bigsqcup_{i\geq 0} \G/(\G_i \cap \G_{i+1}).
\end{align*}
Defined in this way, $X$ is the levelled coset graph for $\G_0,\G_1,\G_2,\dots\le \G$. We can now describe 
the vertices and edges of $X_g = \Ga(g) \,\backslash X$ as follows:
\begin{align*}
V(X_g) &  =  \bigsqcup_{i\geq 0} \G(g) \backslash (\G/\G_i), &
%&&    \bigsqcup_{i\geq 0} \G_i \cdot \G/\G(g) \\
E(X_g) &  =  \bigsqcup_{i\geq 0} \G(g) \backslash (\G/(\G_i \cap \G_{i+1})).
%&&  \leftrightarrow  \bigsqcup_{i\geq 0} (\G_i \cap \G_{i+1}) \cdot \G/\G(g) 
\end{align*}
Define groups $H=\G/\G(g)$ and $H_i=\G_i\G(g)/\G(g)$, 
and coset spaces $L_i=H/H_i$. We have 
$\Stab_{\Gamma}(\Lambda_i)=\Gamma_i$ and so $\Stab_{H}(\Gamma(g)\cdot \La_i) = (\G_i\G(g))/\G(g)=H_i.$
Thus we can identify $\G(g) \backslash (\G/\G_i)$ with $L_i$.
Similarly $\G(g) \backslash (\G/(\G_i \cap \G_{i+1}))$ can be identified with
$H/H_{i}\cap H_{i+1}$. 
So $X_g$ can be viewed as the levelled coset graph of
$H_0,H_1,H_2,\dots\le H$.

We can now establish  that  $ H =\Gamma / \G(g) \cong \SL_2(R_g)$ where $R_g=\F_q[t]/(g)$.
The argument is the same as in \cite{Sh} for $\SL_2(\Z)$.  %We show that the sequence 
%$$ 1 \longrightarrow \Gamma(g) \longrightarrow \Gamma \longrightarrow 
%\SL_2(\F[t]/(g)) \longrightarrow 1$$
%is exact.  This follows easily from the following proposition.
\begin{prop}
The map $\SL_2(\F_q[t]) \rightarrow \SL_2(R_g)$ given by 
$A \mapsto A\bmod(g)$ is surjective.
\end{prop}

\subsection{The structure of $X_g$}
Note that we could use Proposition~\ref{P-gen} to prove that $X_g$ is connected, but this already
follows from the fact that $X_g$ is a quotient of a connected graph.

Write $g= \prod_{i=1}^s g_i^{n_i}$ where the $g_i$ are distinct irreducible polynomials with $\deg(g_i)=d_i$ and $\sum_in_id_i=n$.
Then
$$
  R_g \cong \bigoplus_{i=1}^s R_i \quad\text{where}\quad 
  R_i := R_{g_i^{n_i}}\cong   
  \F_{q^{d_i}}[t_i]/(t_i^{n_i}).
$$
%Note that every element of $R_i$ can be written as $a_0+a_1t_i+...+a_{n_i}t_i^{n_i-1}$.
By Corollary 2.4 of \cite{Han},
$$
  R_g^\times \cong \prod_iR_i^\times  \quad\text{and}\quad
  \GL_2(R_g) \cong \prod_i GL_2(R_i^\times).
$$ 
Using Theorem 2.7(3) of \cite{Han} we get $ |H|=|\SL_2(R_g)|=\frac{|\GL_2(R_g)|}{|R^\times|}= q^{3n}\Pi(q),$
where \linebreak $\Pi(q):= \prod_i\left(1-\frac1{q^{2d_i}}\right)$. Now $H_{n-1}=H_n=H_{n+1}=\cdots$, and so $X_g$ is a bipartite graph which may be described as a collection of disjoint infinite rays beginning at each vertex of level $L_{n-1}$.
It suffices to describe the graph induced by levels $0$ through $n-1$.
For $i\leq n-1$, we have $\G_i \cap \G(g) = \{1\}$, so $H_i\cong \Gamma_i$.
Now $H_0=\SL_2(q)$, and $H_i$ is a semidirect product of 
$(\F_q^+)^{\min(n,i+1)}$ by $\F_q^\times$.
So we have formulas for the number of vertices in each level:
\begin{equation}
|L_i| = 
\begin{cases}
  q^{3n-3}\,\Pi(q) \left(1-\frac1{q^{2}}\right)^{-1}& \quad\text{for $i=0$,}\\
  q^{3n-2-i}\,\Pi(q)\left(1-\frac1{q}\right)^{-1} & \quad\text{for $0< i < n$,}\\
  q^{2n-2}\,\Pi(q)\left(1-\frac1{q}\right)^{-1} & \quad\text{for $i\geq n$.}\\
\end{cases}\notag
\end{equation}

\begin{remark}\label{rem}$\;$
\begin{enumerate}
\item The edges run between consecutive levels, with
the edges between $L_i$ and $L_{i+1}$ in the orbit  of the edge $\La_i\rightarrow\La_{i+1}$ in $X$.  
\item The subgraph induced by $L_0$ and $L_1$ is a $(q+1,q)$-regular bipartite graph.  
\item For $i=1,...,n-1$, each vertex in $L_i$ has $q$ edges to vertices in $L_{i-1}$ and only $1$ edge to a vertex in $L_{i+1}$.  For $i\geq n$, each vertex in $L_i$ has one edge to $L_{i-1}$ and one edge to $L_{i+1}$. 
So there is a semi-infinite ray, also called a cusp, attached to each vertex in $L_{n-1}$.  
\end{enumerate}
\end{remark}

\noindent We have
\begin{equation}
\Stab_{\Gamma(g)}(\Lambda_i)=\Gamma_i \cap \Gamma(g) =  
\begin{cases}
\{1\} & \quad\text{if } i<n\\
U_i = \left\{\left(\begin{smallmatrix} 1 & gf\\ 
0 & 1\end{smallmatrix}\right) \mid f\in \mathbb{F}_q[t],\ \deg(f)\leq i-n 
 \right\} & \quad\text{if } i\geq n \\
\end{cases}\notag
\end{equation}

The stabilizer of any vertex in $L_i$ is then conjugate to 
$\Gamma_i \cap \Gamma(g)$.  Thus the `core' vertices in the graph of groups are labeled 
with the trivial group, and the `cusp' vertex groups along each ray 
are of the form 
$s_j U_i s_j^{-1} , $  where 
$\{s_j \mid j=1, \ldots , k=(q+1)q^{2(n-1)} \}$ 
is a set of conjugacy class representatives.

\subsection{Detailed examples of fundamental domains for congruence subgroups}
In this subsection we construct certain specific examples of the graph $X_g$ for 
the congruence subgroups of $\SL_2$.
%\begin{enumerate}
%\item 
When $g$ is linear, we have $|L_0| = 1$ and $|L_i|=q+1$ for $i\geq 1$.  Thus $X_g$ consists of a single core vertex plus
$q+1$ cusps which are semi-infinite rays.

%\item 
Let $g(t)=t^2$.  Then $|L_0|=q^3$ and $|L_i|=(q+1)q^2$ 
for $i\geq 1$.  The first two levels form a $(q+1,q)$-regular bipartite graph, and semi-infinite rays are attached to each vertex in level $L_1$.
The graph $X_g$ for $q=2$ is given in Figure~\ref{F-SL22}.
The odd and even levels of vertices give the bipartition of Remark~\ref{rem}(1).

  \begin{figure}[!ht]
   \centering
   \includegraphics[scale=.6]{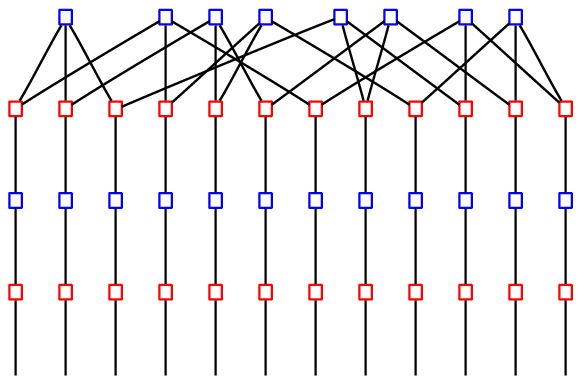}
   \caption{$X_g$ for $g(t)=t^2$, $q=2$}\label{F-SL22}
  \end{figure}

%\item 
Let $g(t)=t^3$.  Here, 
$|L_0|=q^6$, $|L_1|=(q+1)q^5$ and 
$|L_i|=(q+1)q^4$ 
for $i\geq 2$.  The bipartite graph between the first two levels is 
$(q+1,q)$-regular, and then the graph collapses once by a factor of 
$q$ before extending onward as infinite rays.  The core graph for $q=2$ is given in Figure~\ref{F-SL32}, with the rows of vertices top to bottom corresponding to $L_0$, $L_1$ and $L_2$, respectively.  

  \begin{figure}[!ht]
   \centering
   \includegraphics[scale=.2]{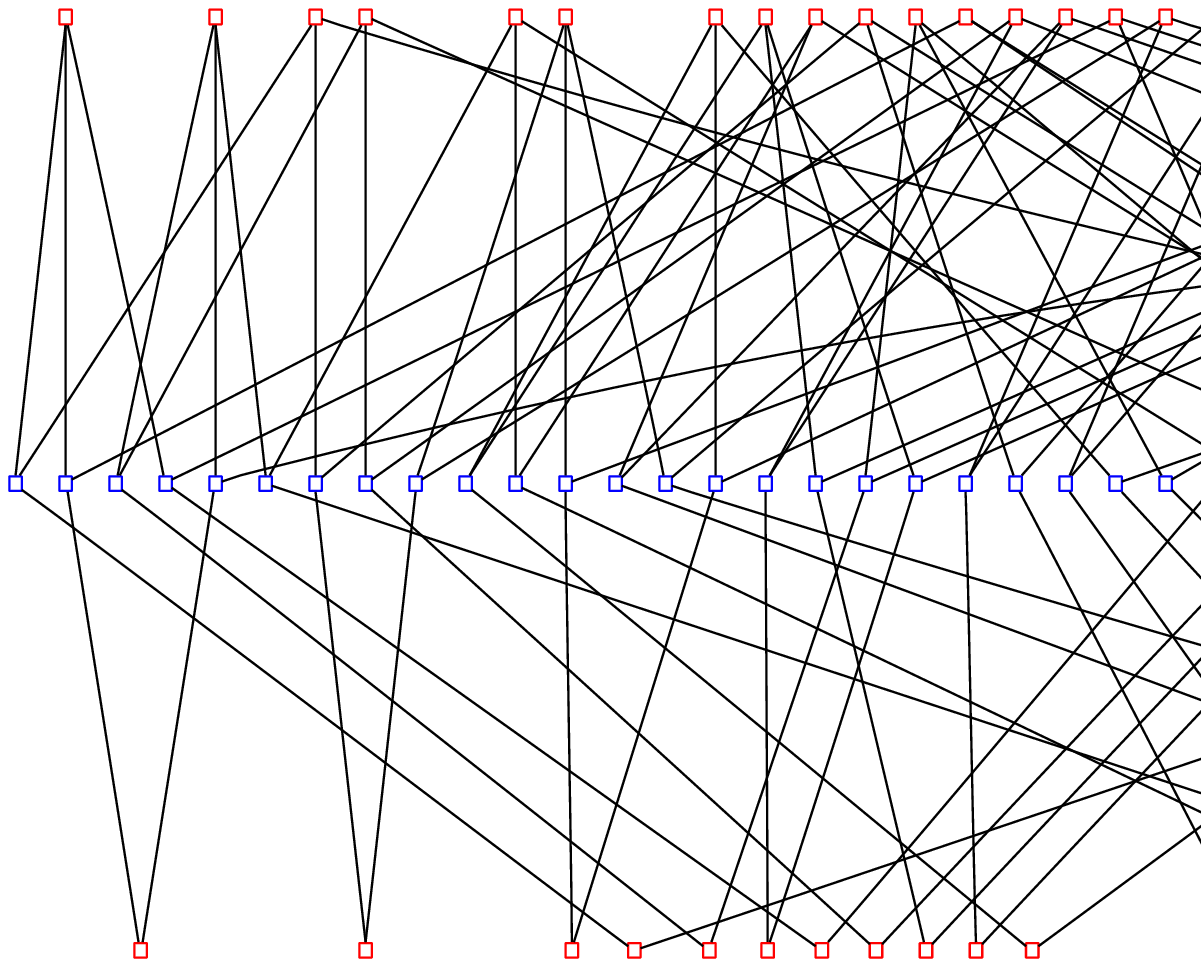}
   \caption{Core of $X_g$ for $g(t)=t^3$, $q=2$}\label{F-SL32}
  \end{figure}

%\end{enumerate}

We used Magma to construct these graphs.
The groups $H$ and $H_i$ are constructed as matrix groups of degree
$2n$ over $\F_q$, and then the coset graphs are constructed using code  due to Leemans \cite{Leemans}.
We used {\tt dot} to draw Figures~\ref{F-SL22} and~\ref{F-SL32} \cite{graphviz}.
%Due to the large size of the core graph it is impractical to draw the larger examples generated by Magma, but a database of all our examples can be made available on request.

\section{Fundamental domains for congruence Subgroups of $\PGL_2(\F_q[t])$}

In \cite{M},   Morgenstern's motivation was to provide the first known examples of linear families of bounded concentrators. We prove however that, in characteristic 2, Morgenstern's  constructions  yield graphs that are not connected. The main source of Morgenstern's error was his incorrect assumption that $\Gamma/\Gamma(g)\cong\PGL_2(R_g)$ where $R_g=\F_q[t]/(g)$. The correct formula for $\Gamma/\Gamma(g)$ is somewhat more complicated and is given in this section. We denote the corrected graphs for $\PGL_2$ by $\overline{X}_g$, and Morgenstern's incorrect coset construction by $\widetilde{X}_g$ (in \cite{M}, both are denoted $X_g$).
%e prove that Morgenstern's full graphs $\widetilde{X}_g$ are connected in odd characteristic. 
%When Morgenstern's  fundamental domains for congruence subgroups as ramified coverings are not connected, we show that all connected components of the graphs are isomorphic and that there is a group acting freely by permuting the components. 

Let $\overline\G = \PGL_2(\F_q[t])$ and let $\overline\G (g) = \{ A \in \overline\G \mid A\equiv I_2\bmod(g) \}$.
Let $\overline{X}_g$ be the graph defined for $\PGL$ in the analogous manner to the graph $X_g$ from the previous section.

First we describe the structure of $\overline{H} := \overline\G /\,\overline\G(g)$. The proof is straightforward.
\begin{prop}
$\overline{H} \cong (\SL_2(R_g) \rtimes F )/ Z$ where
$F=\left\{\left(\begin{smallmatrix} a&0\\0&1\end{smallmatrix}\right)\mid a\in\F_q^\times\right\}$ and $Z=\F_q^\times I_2$.
\end{prop}

%Note that $|\overline{H}|=|H|$ and the subgroup $\overline{H}_i:=H_iF/Z$ has the same size as the corresponding subgroups $H_i$ for $\SL_2$.Hence the vertex levels also have the same sizes.

\begin{theorem}
The $\PGL_2$ graph $\overline{X}_g$ is isomorphic to the $\SL_2$ graph $X_g$.
\end{theorem}
\begin{proof}
We define a map $\phi$ from the vertices of $X_g$ to the vertices of $\overline{X}_g$ by $
  H_ix \mapsto H_iFx/Z. $ Note that $H_iF=FH_i$ for all $i$.
Recall that the edge between $H_ix$ and $H_{i+1}x$ corresponds to the coset $(H_i\cap H_{i+1})x$.
Similarly the edge between $H_iFx/Z$ and $H_{i+1}Fx/Z$ corresponds to the coset $(H_iF\cap H_{i+1}F)x/Z$.
So to prove that $\phi$ takes every edge to an edge it suffices to show that \linebreak
$H_iF\cap H_{i+1}F= 
(H_i\cap H_{i+1}).$ 
Clearly $(H_i\cap H_{i+1})\subseteq H_iF\cap H_{i+1}F$.
Conversely suppose $hf=kg$ for $h\in H_i$, $k\in H_{i+1}$, $f,g\in F$.
Then $f=\diag(a,1)=g$ where $a=\det(hf)=\det(kg)$, and so $h=k\in H_i\cap H_{i+1}$.
Finally we can conclude that $\phi$ is an isomorphism since the number of edges at level $i$ is the same for the two graphs.
\end{proof}
In particular, $\overline{X}_g$ is always connected, unlike the graph $\widetilde{X}_g$ constructed in \cite{M}.

\section{Morgenstern's graphs}

\subsection{Morgenstern's $\PGL$ graph}
%Let $\Gamma = PGL(\mathbb{F}_q[t])$.  For a polynomial $g\in \mathbb{F}_q[t]$, let $\Gamma (g) = \{ A \in \Gamma : A \equiv I \text{ mod } g \}$.   Both $\Gamma$ and $\Gamma(g)$ act on the $q+1-\text{regular tree} X = X_{q+1}$.  In Morgenstern's paper, he gives the structure for a graph $X_g'$ as follows.
%Let $degree(g)=n$ and $R_{g}=\mathbb{F}_q[t]/g$. Define the groups $H_0=SL(2,\mathbb{F}_q)$ and for $1\leq i \leq n-1$, $H_i= T + (\mathbb{F}_q+\mathbb{F}_qt+ ... + \mathbb{F}_qt^i) E_{1,2}$ where $T$ is the group of all diagonal matrices in $SL_2(\F_q)$, and $E_{1,2}$ is an elementary matrix, as in the previous chapter.  
%For all $i\geq n$, let $H_i=H_{n-1}$.  

Let $\widetilde{H}=\PGL_2(R_g)=\GL_2(R_g)/\widetilde{Z}$, where 
$\widetilde{Z}=R_g^\times I_2$.   
Let $\widetilde{H}_i$ be the subgroup 
$H_iF\widetilde{Z}/\widetilde{Z}$, and define levels
$\widetilde{L}_i = \PGL_2(R_g)/\widetilde{H}_i$.
Morgenstern's graph $\widetilde{X}_g$ is now defined as the levelled coset graph for $\widetilde{H}_0,\widetilde{H}_1,\dots$ in $\widetilde{H}$.
This is analogous to the constructions of $X_g$ in Section~3.1 and
$\overline{X}_g$ in Section~4.
Furthermore 
$$|H|=|\overline{H}|=|\widetilde{H}|, \qquad
|H_i|=|\overline{H}_i|=|\widetilde{H}_i|, \qquad
|H_i\cap H_{i+1}|=|\overline{H}_i\cap\overline{H}_{i+1}|=|\widetilde{H}_i\cap\widetilde{H}_{i+1}|$$
for all $i\ge 0$.
Hence the properties of Remark~\ref{rem} hold for all three graphs.
We have already seen that $X_g\cong \overline{X}_g$.
Morgenstern claims that the graphs $\overline{X}_g$ and 
$\widetilde{X}_g$ are isomorphic, but we will see that this is not always the case.  This is a consequence of the fact that Morgenstern fails to prove that he has the desired ramified covering. We now consider connectedness properties of $\widetilde{X}_g$.
\begin{prop}\label{prop-Xtildeconn}
Morgenstern's graph $\widetilde{X}_g$ has 
$|R_g^\times:\F_q^\times R_g^{\times2}|$ connected components, where \linebreak $R_g^{\times2}=\{x^2\mid x\in R_g^{\times}\}$.
\end{prop}
\begin{proof}
By the connectedness of $X_g$ and  Proposition~\ref{P-gen}, we know $\langle H_0,H_{n-1}\rangle=H$.
Hence
\begin{align*}
  \langle \widetilde{H}_0,\widetilde{H}_{n-1}\rangle
  &= \langle H_0F\widetilde{Z}, \widetilde{H}_{n-1}F\widetilde{Z}\rangle/\widetilde{Z} \\
  &= \langle H_0, {H}_{n-1}\rangle F\widetilde{Z}/\widetilde{Z} = HF\widetilde{Z}/\widetilde{Z}.
\end{align*}
Since $\det$ maps $\GL_2(R_g)$ onto $R_n^\times$ with kernel $H$, we have 
\begin{align*}
\GL_2(R_g)/HF\widetilde{Z} \cong \F_q^\times R_n^\times/\det(F\widetilde{Z})&=R_n^\times/\F_q^\times R_n^{\times2}.\qedhere
\end{align*}
\end{proof}

\begin{lemma}\label{L-square}
Let $R = \E[u]/(u^n)$ where $\E:=\F_{q^d}$.
\begin{enumerate}
\item If $q$ is odd, then $R^{\times2} = \E^{\times2}+\E u+\E u^2+\cdots$ and so $\E^\times R^{\times2} =R^\times$.
\item If $q$ is even, then $R^{\times2} =\E^\times R^{\times2} = \E^\times+\E u^2+\E u^4+\cdots$.
\end{enumerate}
\end{lemma}
\begin{proof}
For $q$ even, 
$(a_0+a_1u+a_2u^2+\cdots)^2 = a_0^2+a_1^2u^2+a_2^2u^4+\cdots$,
for all $a_i\in\E$.
Using the fact that $\E^{\times2}=\E^\times$, we get
$
  R^{\times2}= \E^\times+\E u^2+\E u^4+\cdots.
$

Now let $q$ be odd. It suffices to show that every element of the form
$1+a_1u+\cdots$ is in $R^{\times2}$.
Suppose this is not true, and take $a=1+a_iu^i+\cdots\notin R^{\times2}$ with $i$ maximal such that $a_i\ne0$.
But $R^{\times2}$ is a subgroup of $R^\times$, and so
$a(1-\frac{a_i}2u)^2\notin R^{\times2}$.
Since the coefficients of $u,u^2,\dots,u^i$ are all zero in this element,
we have a contradiction.
\end{proof}

\begin{theorem}
Morgenstern's graph $\widetilde{X}_g$ is connected
if and only if $q$ is odd or $g$ is squarefree.
\end{theorem}
\begin{proof}
This follows immediately from the previous two results
and the decomposition\linebreak
$R_g\cong \bigoplus_{i=1}^s  
  \F_{q^{d_i}}[t_i]/(t_i^{n_i})$.
\end{proof}
In particular, $\widetilde{X}_g$ is not isomorphic to $\overline{X}_g$
when $q$ is even and $g$ is not squarefree.
By Magma computation using the algorithm of \cite{nauty}, we found that $X_{t^n}$ and $\widetilde{X}_{t^n}$ are also nonisomorphic for $q=3$ and $n=2,3,4$.  
%For $q$ odd, 
%suppose $a=a_0+a_1t+\cdots\in R_g$ with $a_0\ne\pm1$.
%Then $(a\pm1)/2$ is invertible and so
%$$$\overline{X}_g$
%a=\left(\frac{a+1}{2}\right)^2-\left(\frac{a-1}{2}\right)^2.
%$$
%If $a_0=\pm1$, then $a_0$ is a square, and $a-a_0$ is a difference of %squares as above.

\subsection{The subgraphs of levels $0-1$}
Morgenstern constructed $\widetilde{X}_g$ as a means of providing examples of linear families of bounded concentrators.  These examples were obtained as the subgraph $\widetilde{D}_g(0-1)$ induced by the vertices of $\widetilde{X}_g$ in the first two levels $\widetilde{L}_0$ and $\widetilde{L}_1$.  However, a necessary property for a bounded concentrator is connectedness.  We will show in characteristic~2 that the subgraphs $\widetilde{D}_g(0-1)$ are  not connected.  This contradicts the following claim of Morgenstern:
\begin{quote}\emph{\cite{M}, Proposition 4.2}: 
If $q\geq 4$, or $q=3$ and $g(x)$ is irreducible of
degree greater than 2, then $\widetilde{D}_g(0-1)$ is connected.
\end{quote}
This in turn is based on an incorrect lower bound for $N_0(S)$, the set of vertices in 
$\widetilde{L}_0$ which are adjacent to a subset 
$S \subseteq \widetilde{L}_1$ of vertices in $\widetilde{L}_1$:
\begin{quote}\emph{\cite{M}, Lemma 4.1}:
For every $S \subseteq \widetilde{L}_1$, $\frac{|N_0(S)|}{|S|} \geq \frac{q|\widetilde{L}_1|}{(q-3)|S|+4|\widetilde{L}_1|}$.
\end{quote}
This bound fails if we take $S$ to be a connected component of one of the disconnected graphs described below.
We believe that these two results are correct when applied to the correct fundamental domain $\overline{X}_g$ for 
$\PGL_2$ described in Section~4.
We note that, when $\widetilde{D}_g(0-1)$ is not connected,
all the connected components are isomorphic. 
Furthermore $H$ acts transitively on the set of components.
This follows from general properties of coset graphs.

%ed (our size of $L_1$ matches the formula given by Morgenstern).  We speculate that this bound will fail in general for characteristic $2$.  It is interesting to note that we first tried a random search of the subsets of $L_1$ but did not find a counterexample.  It is possible that the bound holds for most subsets of $L_1$.  We do not know what happens to this bound in odd characteristic.

%We remark that our interpretation of Morgenstern's result is supported by the following.  First, there is no condition given on the polynomial $g(x)$ in Lemma 4.1 of [M].  The mention of irreducibility is not made until [M], Proposition 4.2, after the error has already occured.  Morevover the proof of [M], Proposition 4.2, given by Morgenstern, clearly relies on the irreducibility hypothesis only in the case $q=3$.  Once again, the result should hold regardless of the polynomials used.

In the remainder of this section we consider connectedness properties of $\widetilde{D}_g(0-1)$ and the corresponding subgraph 
$D_g(0-1)$ induced on the first two levels of $X_g$ (or equivalently
$\overline{X}_g$).
By Proposition~\ref{P-gen}, the number of components of $D_g(0-1)$ is
$$ C := |H: \langle H_0,H_1\rangle|,$$
and the number of components $\widetilde{D}_g(0-1)$ is
$$ \widetilde{C} := |\widetilde{H}: \langle \widetilde{H}_0,\widetilde{H}_1\rangle| = |\GL_2(R_g): \langle H_0,H_1\rangle F\widetilde{Z}|.$$
This allows us to count components using Magma's matrix group machinery.
These results, for even $q$ and $g(t)=t^n$, are summarised in Table~\ref{T-comp}.
For odd $q$ we found both graphs to be connected in every example we computed.
\begin{table}[ht!]
\begin{tabular}{r|rrrrrrrrrrrrrrrrrrrrrrrrrr}
$q$                               &2\\
$n$                               &2  &3  &4  &5  &6  &7   &8   &9   &10  &11  &12  &13  &14  \\
C &$1$&$2^2$&$2^3$&$2^5$&$2^6$&$2^8$&$2^9$&$2^{11}$&$2^{12}$&$2^{14}$&$2^{15}$&$2^{17}$&$2^{18}$\\
$\widetilde{C}$      &$2^1$&$2^3$&$2^4$&$2^6$&$2^7$&$2^10$&$2^{11}$&$2^{13}$&$2^{14}$&$2^{17}$&$2^{18}$&$2^{20}$&$2^{21}$&\\
\hline
$q$&2\\
$n$&15  &16  &17  &18  &19  &20  &21  &22  &23  &24  &25  &26\\
C&$2^{20}$&$2^{21}$&$2^{23}$&$2^{24}$&$2^{26}$&$2^{27}$&$2^{29}$&$2^{30}$&$2^{32}$&$2^{33}$&$2^{35}$&$2^{36}$\\
$\widetilde{C}$ &$2^{24}$&$2^{25}$&$2^{27}$&$2^{28}$&$2^{31}$&$2^{32}$&$2^{34}$&$2^{35}$&$2^{38}$&$2^{39}$&$2^{41}$&$2^{42}$\\
\hline
$q$ & 4&&&&  &   &   & &&&&& \\
$n$ & 2&3&4&5&6  &7   &8   &9   &10  &11  &12  &13 \\
C&1&1&1&1&1&1&1&1&1&1&1&1\\
$\widetilde{C}$ &$2^2$&$2^2$&$2^4$&$2^4$&$2^6$&$2^6$&$2^8$&$2^8$&$2^{10}$&$2^{10}$&$2^{12}$&$2^{12}$\\
\hline
$q$&8&&&&&16&&&&32&&64\\
$n$&2  &3  &4  &5  &6  &7 &2&3&4&2&3&2\\
C&1&1&1&1&1&1&1&1&1&1&1&1\\
$\widetilde{C}$ &$2^3$&$2^3$&$2^6$&$2^6$&$2^9$&$2^9$&$2^4$&$2^4$&$2^8$&$2^5$&$2^5$&$2^6$\\
\end{tabular}\vspace{2mm}
\caption{Number of components of the first two levels for $q$ even}
\label{T-comp}
\end{table}

Based on these experimental results, we conjecture formulas:
\begin{conjecture} For $g(t)=t^n$ over $\F_q$,
\begin{align*}
  C&=
    \begin{cases}
      q^{\lfloor(3n-5)/2\rfloor} \quad\quad\quad&\quad\text{for $q=2$, $n>2$,}\\
      1 &\quad\text{for $q>2$,} 
    \end{cases}\\
  \widetilde{C}&=
    \begin{cases}
      q^{\lfloor(3n-5)/2\rfloor+\lfloor(n+1)/4\rfloor} &\quad\text{for $q=2$, $n>2$,}\\
      q^{\lfloor n/2\rfloor} &\quad\text{for $q>2$ even, $n>1$,}\\
      1 &\quad\text{for $q$ odd.} 
    \end{cases}
\end{align*}
\end{conjecture}

We  now give some theoretical results on the number of components  for arbitrary $g$.

\begin{prop}
$$ 
  C \cdot |R_g^\times:\F_q^\times R_g^{\times2}| = 
\widetilde{C} \cdot |S:T|
$$
where $S := \{a \in R_g^\times \mid a^2\in\F_q^\times \}$ and 
$T:= \left\{ a\in S \mid \left(\begin{smallmatrix}a^{-1}&0\\0&a\end{smallmatrix}\right)
\in \langle H_0,H_1\rangle\right\}$.
\end{prop}
\begin{figure}
\setlength{\unitlength}{1500sp}%
\begingroup\makeatletter\ifx\SetFigFont\undefined%
\gdef\SetFigFont#1#2#3#4#5{%
  \reset@font\fontsize{#1}{#2pt}%
  \fontfamily{#3}\fontseries{#4}\fontshape{#5}%
  \selectfont}%
\fi\endgroup%
\begin{picture}(4230,6786)(3136,-7051)
\thinlines
{\color[rgb]{0,0,0}\put(7201,-361){\line(-1,-1){3000}}
\put(4201,-3361){\line( 1,-1){1200}}
\put(5401,-4561){\line( 1, 1){1800}}
\put(7201,-2761){\line(-1, 1){1200}}
}%
{\color[rgb]{0,0,0}\put(5401,-4561){\line(-1,-1){1200}}
\put(4201,-5761){\line( 1,-1){1200}}
\put(5401,-6961){\line( 1, 1){1200}}
\put(6601,-5761){\line(-1, 1){1200}}
}%
\put(6300,-436){\makebox(0,0)[lb]{\smash{{\SetFigFont{12}{14.4}{\rmdefault}{\mddefault}{\updefault}{\color[rgb]{0,0,0}{\small $\qquad\;\;\;\GL_2(R_g)$}}%
}}}}
\put(4700,-1411){\makebox(0,0)[lb]{\smash{{\SetFigFont{12}{14.4}{\rmdefault}{\mddefault}{\updefault}{\color[rgb]{0,0,0}{\small $HF\widetilde{Z}\qquad\qquad$}}%
}}}}
\put(1500,-3436){\makebox(0,0)[lb]{\smash{{\SetFigFont{12}{14.4}{\rmdefault}{\mddefault}{\updefault}{\color[rgb]{0,0,0}{\small $\langle H_0,H_1\rangle F\widetilde{Z}\qquad\qquad\qquad$}}%
}}}}
\put(7351,-2836){\makebox(0,0)[lb]{\smash{{\SetFigFont{12}{14.4}{\rmdefault}{\mddefault}{\updefault}{\color[rgb]{0,0,0}{\small $H$}}%
}}}}
\put(2200,-5836){\makebox(0,0)[lb]{\smash{{\SetFigFont{12}{14.4}{\rmdefault}{\mddefault}{\updefault}{\color[rgb]{0,0,0}{\small $\langle H_0,H_1\rangle$}}%
}}}}
\put(5701,-4636){\makebox(0,0)[lb]{\smash{{\SetFigFont{12}{14.4}{\rmdefault}{\mddefault}{\updefault}{\color[rgb]{0,0,0}{\small $H\cap(\langle H_0,H_1\rangle F \widetilde{Z})=(H\cap F\widetilde{Z})\langle H_0,H_1\rangle$}}%
}}}}
\put(6826,-5836){\makebox(0,0)[lb]{\smash{{\SetFigFont{12}{14.4}{\rmdefault}{\mddefault}{\updefault}{\color[rgb]{0,0,0}{\small $H\cap F\widetilde{Z}$}}%
}}}}
\put(5776,-7036){\makebox(0,0)[lb]{\smash{{\SetFigFont{12}{14.4}{\rmdefault}{\mddefault}{\updefault}{\color[rgb]{0,0,0}{\small $\langle H_0,H_1\rangle\cap F\widetilde{Z}$}}%
}}}}
\end{picture}%
\caption{Subgroup lattice}\label{F-lattice}
\end{figure}
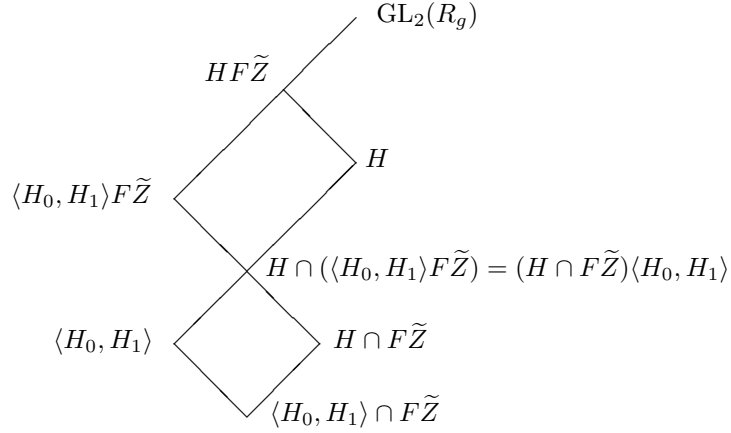
\begin{proof}
From Figure~\ref{F-lattice}, we can see that
$$
  C \cdot |\GL_2(R_g):HF\widetilde{Z}| = \widetilde{C} \cdot |H\cap F\widetilde{Z}:\langle H_0,H_1\rangle \cap F\widetilde{Z}|
$$
Since $\det$ maps $G$ onto $R_n^\times$ with kernel $H$, we have 
$\GL_2(R_g)/HF\widetilde{Z} \cong R_g^\times/\det(F\widetilde{Z})=R_g^\times/\F_q^\times R_g^{\times2}$.
An element of $F\widetilde{Z}$ has the form $x=\left(
\begin{smallmatrix} \lambda a & 0 \\ 0 & a \end{smallmatrix} \right)$, for
$\lambda\in\F_q^\times$ and $a\in R_g^\times$.  And $x\in H$ is equivalent to
$a^2=\lambda^{-1}\in\F_q^\times$, so projection onto the bottom right entry
gives an isomorphism from  $H\cap F\widetilde{Z}$ to $S$.
Clearly the subgroup $\langle H_0,H_1\rangle \cap F\widetilde{Z}$
corresponds to the $T$ under this isomorphism.
\end{proof}

\begin{prop}
If $q$ is odd and $g(t)=t^n$, then $C=\widetilde{C}$.
\end{prop}
\begin{proof}
We have $\F_q^\times R_g^{\times2}=R_g^\times$ by Lemma~\ref{L-square}.
If $a=a_0+a_it^i+\cdots\in S$ with $a_i$ the smallest nonzero coefficient other than $a_0$, then $a^2=a_0+2a_it^i+\cdots=1$ and so
$i\ge n$.  Hence $S=\F_q^\times$, and it is now easy to prove that $T=S$.
\end{proof}
%[Can we prove $C=1$?]

\begin{prop}
If $q$ is even and $g$ is not squarefree, then $\widetilde{C}>C$.
\end{prop}
\begin{proof}
By Lemma~9 and the decomposition $R=\bigoplus_r R_i$,
we get $|R_g^\times:\F_q^\times R_g^{\times2}|= \prod_iq^{d_i\lfloor n_i/2 \rfloor}.$
Now suppose $a=a_0+a_1t_i+a_2t_i^2+\cdots\in R_i$ with $a^2=1$.
This is equivalent to $a_0=1$, and $a_i=0$ for all $j>0$ with $2j<n_i$.
Hence $|S|=\prod_iq^{d_i\lfloor n_i/2 \rfloor}$.

We now have $\widetilde{C}=|T|C$.
But if $2^e<n_i\le 2^{e+1}$, then $a=1+t_i^{2^e}$ is a nontrivial element which squares to the identity.
And $\langle H_0,H_1\rangle$ contains
$$
\left(\begin{matrix} a&0\\0&a\end{matrix}\right) =
\left[\left(\begin{matrix} 1&a\\0&1\end{matrix}\right)
\left(\begin{matrix} 0&1\\1&0\end{matrix}\right)\right]^3,
%\left(\begin{matrix} 1&a\\0&1\end{matrix}\right)
%\left(\begin{matrix} 0&1\\1&0\end{matrix}\right)
%\left(\begin{matrix} 1&a\\0&1\end{matrix}\right)
%\left(\begin{matrix} 0&1\\1&0\end{matrix}\right).
$$
so $T$ is nontrivial.
\end{proof}

So, for $q$ even and $g$ not squarefree, we know
that $\widetilde{D}_g(0-1)$ is not connected, and also that it cannot be isomorphic to $D_g(0-1)$.
By Magma computation using the algorithm of \cite{nauty}, we found that $D_{t^n}(0-1)$ and $\widetilde{D}_{t^n}(0-1)$ are also nonisomorphic for $q=3$ and $n=2,3,4$.  However they are isomorphic for $q=5,7$ and $n=2$.  

%What remains to be seen is whether or not the subgraphs $D_g(0-1)$ are connected in odd characteristic and if they have the claimed expansion properties.  This is beyond the scope of the current work.
\nocite{B}
\bibliographystyle{alpha}

\bibliography{ccm}

\def\cprime{$'$}
\begin{thebibliography}{McK81}

\bibitem[Bas93]{B}
Hyman Bass.
\newblock Covering theory for graphs of groups.
\newblock {\em J. Pure Appl. Algebra}, 89(1-2):3--47, 1993.

\bibitem[BC97]{magma}
W.~W. Bosma and J.~J. Cannon.
\newblock {\em Handbook of Magma functions}.
\newblock School of Mathematics and Statistics, University of Sydney, Sydney,
  1997.

\bibitem[But07]{Bu}
Ralf Butenuth.
\newblock {\em {E}in {A}lgorithmus zum {B}erechnen von {H}ecke-{O}peratoren auf
  {D}rinfeldschen {M}odulformen}.
\newblock Diplom Thesis, Universit\"at Duisburg-Essen, 2007.

\bibitem[Dri77]{D}
V.~G. Drinfel{\cprime}d.
\newblock Elliptic modules. {II}.
\newblock {\em Mat. Sb. (N.S.)}, 102(144)(2):182--194, 325, 1977.

\bibitem[Geb08]{G}
Max Gebhardt.
\newblock Private communication.
\newblock 2008.

\bibitem[Gek80]{Ge1}
Ernst-Ulrich Gekeler.
\newblock {\em Drinfeld-{M}oduln und modulare {F}ormen \"uber rationalen
  {F}unktionenk\"orpern}.
\newblock Bonner Mathematische Schriften [Bonn Mathematical Publications], 119.
  Universit\"at Bonn Mathematisches Institut, Bonn, 1980.
\newblock Dissertation, Rheinische Friedrich-Wilhelms-Universit{\"a}t, Bonn,
  1979.

\bibitem[Gek85]{Ge2}
E.-U. Gekeler.
\newblock Automorphe {F}ormen \"uber {${\bf F}_q(T)$} mit kleinem {F}\"uhrer.
\newblock {\em Abh. Math. Sem. Univ. Hamburg}, 55:111--146, 1985.

\bibitem[GN95]{GN}
Ernst-Ulrich Gekeler and Udo Nonnengardt.
\newblock Fundamental domains of some arithmetic groups over function fields.
\newblock {\em Internat. J. Math.}, 6(5):689--708, 1995.

\bibitem[GN00]{graphviz}
Edmend~R. Ganser and Stephen~C. North.
\newblock An open graph visualization system with applications to software
  engineering.
\newblock {\em Software--Practice and Experience}, 30(11):1203--1233, September
  2000.

\bibitem[Hal03]{Ha}
Chris Hall.
\newblock Fundamental domains of some drinfeld modular curves.
\newblock Unpublished, 2003.

\bibitem[Han06]{Han}
Juncheol Han.
\newblock The general linear group over a ring.
\newblock {\em Bull. Korean Math. Soc.}, 43(3):619--626, 2006.

\bibitem[JL04]{Leemans}
Pascale Jacobs and Dimitri Leemans.
\newblock An algorithmic analysis of the intersection property.
\newblock {\em LMS J. Comput. Math.}, 7:284--299 (electronic), 2004.

\bibitem[McK81]{nauty}
B.~D. McKay.
\newblock {Practical Graph Isomorphism}.
\newblock {\em Congressus Numerantium}, 30:45--87, 1981.

\bibitem[Mor95]{M}
Moshe Morgenstern.
\newblock Natural bounded concentrators.
\newblock {\em Combinatorica}, 15(1):111--122, 1995.

\bibitem[Rus98]{Ru}
Imke Rust.
\newblock Arithmetically defined representations of groups of type {${\rm
  SL}(2,\mathbb F\sb q)$}.
\newblock {\em Finite Fields Appl.}, 4(4):283--306, 1998.

\bibitem[Ser03]{S}
Jean-Pierre Serre.
\newblock {\em Trees}.
\newblock Springer Monographs in Mathematics. Springer-Verlag, Berlin, 2003.
\newblock Translated from the French original by John Stillwell, Corrected 2nd
  printing of the 1980 English translation.

\bibitem[Shi94]{Sh}
Goro Shimura.
\newblock {\em Introduction to the arithmetic theory of automorphic functions},
  volume~11 of {\em Publications of the Mathematical Society of Japan}.
\newblock Princeton University Press, Princeton, NJ, 1994.
\newblock Reprint of the 1971 original, Kan{\^o} Memorial Lectures, 1.

\end{thebibliography}

\end{document}